\documentclass[12pt,a4paper]{article}

\usepackage{amsmath,amstext,amssymb,amscd}
\usepackage{graphicx}

\newtheorem{theorem}{Theorem}[section]
\newtheorem{lemma}[theorem]{Lemma}  
\newtheorem{prop}[theorem]{Proposition} 
\newtheorem{coro}[theorem]{Corollary}  

\newcommand{\Xcomment}[1]{}

\makeatletter \@addtoreset{equation}{section} \makeatother

\newenvironment{proof}{\noindent{\bf Proof}\/}%
{\hfill$\qed$\medskip}

\def\qed{ \ \vrule width.1cm height.3cm depth0cm}

\newenvironment{numitem1}{\refstepcounter{equation}\begin{enumerate}%
\item[(\thesection.\arabic{equation})]}{\end{enumerate}}

\newcommand{\refeq}[1]{(\ref{eq:#1})}  

\makeatletter
\renewcommand{\section}{\@startsection{section}{1}{0pt}%
{-3.5ex plus -1ex minus -.2ex}{2.3ex plus .2ex}%
{\normalfont\Large}}
 \makeatother

 \makeatletter
\renewcommand{\subsection}{\@startsection{subsection}{2}{0pt}%
{-3.0ex plus -1ex minus -.2ex}{-1.5ex plus .2ex}%
{\normalfont\normalsize\bf}}
 \makeatother

\newcommand{\SEC}[1]{\ref{sec:#1}}  

\def\Rset{{\mathbb R}}

\def\Zset{{\mathbb Z}}

\def\Ascr{{\cal A}}

\def\Dscr{{\cal D}}
\def\Escr{{\cal E}}

\def\Lscr{{\cal L}}

\def\Pscr{{\cal P}}
\def\Qscr{{\cal Q}}

\def\tilde{\widetilde}

\def\Path{{\rm Path}}

\def\Sign{{\rm sign}}

\begin{document}

\baselineskip=15pt
\parskip=2pt

\title{Basic quadratic identities on quantum minors
}

\author{
Vladimir~I.~Danilov\thanks{Central Institute of Economics and Mathematics of
RAS, 47, Nakhimovskii Prospect, 117418 Moscow, Russia; email:
danilov@cemi.rssi.ru}
  \and
Alexander~V.~Karzanov\thanks{Central Institute of Economics and Mathematics of
RAS, 47, Nakhimovskii Prospect, 117418 Moscow, Russia; email: akarzanov7@gmail.com. Corresponding author.
}
  }

\date{}

\maketitle

 \begin{abstract} This paper continues an earlier research of the authors on universal quadratic identities (QIs) on minors of quantum matrices. We demonstrate  situations when the universal QIs are provided, in a sense, by the ones of four special types
(Pl\H{u}cker, co-Pl\H{u}cker, Dodgson identities and quasi-commutation relations on flag and co-flag interval minors). 
  \medskip

 \smallskip

\emph{MSC-class}: 16T99, 05C75, 05E99
 \end{abstract}


\section{Introduction}  \label{sec:intr}

Let $\Ascr$ be a $\mathbb K$-algebra over a field $\mathbb K$ and let $q\in \mathbb K^\ast$. We deal with an $m\times n$ matrix $X$ whose entries $x_{ij}$
belong to $\Ascr$ and satisfy
the following ``quasi-commutation'' relations (originally appeared
in Manin's work~\cite{man}): for $i<\ell\le m$ and $j<k\le n$,
   \begin{gather}
   x_{ij}x_{ik}=qx_{ik}x_{ij},\qquad x_{ij}x_{\ell j}=qx_{\ell j}x_{ij},
                                              \label{eq:xijkl}\\
   x_{ik}x_{\ell j}=x_{ \ell j}x_{ik}\quad \mbox{and}\quad
    x_{ij}x_{\ell k}-x_{\ell k}x_{ij}=(q-q^{-1})x_{ik}x_{\ell j}.  \nonumber
    \end{gather}

We call such an $X$ a \emph{fine $q$-matrix} over $\Ascr$ and are interested in relations in the corresponding \emph{quantized coordinate ring} (the
algebra of polynomials in the $x_{ij}$ respecting the relations in $\Ascr$), which are viewed as quadratic identities on $q$-\emph{minors} of $X$. Let us start with some terminology and notation. 
 \smallskip

$\bullet$ ~For a positive integer $n'$, the set $\{1,2,\ldots,n'\}$ is denoted
by $[n']$. Let $\Escr^{n,m}$ denote the set of ordered pairs $(I,J)$ such that
$I\subseteq[m]$, ~$J\subseteq[n]$ and $|I|=|J|$; we will refer to such a pair as a \emph{cortege} and may denote it as $(I|J)$. The submatrix of $X$ whose rows and columns are indexed by elements of $I$ and $J$, respectively, is denoted by $X(I|J)$. For
$(I,J)\in\Escr^{m,n}$, the $q$-\emph{determinant} (called the
$q$-\emph{minor}, the \emph{quantum minor}) of $X(I|J)$ is defined as 
  \begin{equation} \label{eq:qminor}
  \Delta_{X,q}(I|J):=\sum\nolimits_{\sigma\in S_k} (-q)^{\ell(\sigma)}
    \prod\nolimits_{d=1}^{k} x_{i_dj_{\sigma(d)}},
    \end{equation}
where the factors  in $\prod$ are ordered from left to right by increasing $d$, and
$\ell(\sigma)$ denotes the \emph{length} (number of inversions) of
a permutation $\sigma$. The terms $X$ and/or $q$ in $\Delta_{X,q}(I|J)$ may
be omitted when they are clear from the context. By definition $\Delta(\emptyset|\emptyset)$ is the unit of $\Ascr$.
  \smallskip

$\bullet$ ~A quantum quadratic identity (QI) of our interest is viewed as
  \begin{equation} \label{eq:q_ident}
  \sum(\Sign_i q^{\delta_i} \Delta_q(I_i|J_i)\, \Delta_q(I'_i|J'_i) \colon i=1,\ldots,N)=0,
  \end{equation}
where for each $i$, ~$\delta_i\in\Zset$, ~$\Sign_i\in\{+,-\}$, and $(I_i|J_i),\, (I'_i|J'_i)\in\Escr^{m,n}$. Note that any pair $(I|J),(I'|J')$ may be repeated in~\refeq{q_ident} many times. We restrict ourselves by merely \emph{homogeneous} QIs, which means that in expression~\refeq{q_ident},
  \begin{numitem1} \label{eq:homogen}
each of the sets $I_i\cup I'_i,\, I_i\cap I'_i,\, J_i\cup J'_i,\, J_i\cap J'_i$ is
invariant of $i$.
  \end{numitem1}
When, in addition, \refeq{q_ident} is valid for all appropriate $\Ascr,q,X$ (with $m,n$ fixed), we say that~\refeq{q_ident} is \emph{universal}. 
\smallskip

In fact, there are plenty of universal QIs. For example, representative classes 
involving quantum flag minors were demonstrated by Lakshmibai and Reshetikhin~\cite{LR} and Taft and Towber~\cite{TT}. Extending earlier results, the authors obtained in~\cite{DK} necessary and sufficient conditions characterizing all universal QIs. These conditions are given in combinatorial terms and admit an efficient verification.

Four special cases of universal QIs play a central role in this paper. They are exposed in (I)--(IV) below; for details, see~\cite[Sects.~6,8]{DK}. 

In what follows, for integers $1\le a\le b\le n'$, we call the set $\{a,a+1,\ldots,b\}$ an \emph{interval} in $[n']$ and denote it as $[a..b]$ (in particular, $[1..n']=[n']$). For disjoint subsets $A$ and $\{a,\ldots,b\}$, we may abbreviate $A\cup\{a,\ldots,b\}$ as $Aa\ldots b$.  Also for $(I|J)\in\Escr^{m,n}$, ~$\Delta(I|J)=\Delta_{X,q}(I|J)$ is called a \emph{flag} (\emph{co-flag}) $q$-minor if $J=[k]$ (resp. $I=[k]$), where $k:=|I|=|J|$. 
 \smallskip
 
(I) \emph{Pl\H{u}cker-type relations with triples}.  ~Let $A\subset[m]$, $B\subset[n]$, $\{i,j,k\}\subseteq[m]-A$, $\ell\in[n]-B$, and let $|A|+1=|B|$ and $i<j<k$. There are several universal QIs on such elements (see a discussion in~\cite[Sect.~6.4]{DK}). One of them is viewed as
  \begin{equation} \label{eq:Pluck}
  \Delta(Aj|B)\Delta(Aik|B\ell)=\Delta(Aij|B\ell)\Delta(Ak|B)+\Delta(Ajk|B\ell)\Delta(Ai|B).
  \end{equation}
In the flag case (when $B=[|B|]$ and $\ell=|B|+1$) this turns into a quantum analog of the classical Pl\H{u}cker relation with a triple $i<j<k$.
 \smallskip
 
(II) \emph{Co-Pl\H{u}cker-type relations with triples}.  ~They are ``symmetric'' to those in~(I). Namely, we deal with $A\subset[m]$, $B\subset[n]$, $\ell\in[m]-A$ and $\{i,j,k\}\subseteq[n]-B$ such that $|A|=|B|+1$ and $i<j<k$. Then there holds:
  \begin{equation} \label{eq:co-Pluck}
  \Delta(A|Bj)\Delta(A\ell|Bik)=\Delta(A\ell|Bij)\Delta(A|Bk)+\Delta(A\ell|Bjk)\Delta(A|Bi).
  \end{equation}

 (III) \emph{Dodgson-type relations}. ~Let $i,k\in[m]$ and $j,\ell\in[n]$ satisfy $k-i=\ell-j\ge 0$. Form the intervals $A:=[i+1..k-1]$ and $B:=[j+1..\ell-1]$. The universal QI which is a quantum analog of the classical Dodgson relation is viewed as (cf.~\cite[Sect.~6.5]{DK})
  \begin{equation} \label{eq:Dodgson}
 \Delta(Ai|Bj)\Delta(Ak|B\ell)  =  \Delta(Aik|Bj\ell)\Delta(A|B)+ q\Delta(Ai|B\ell)\Delta(Ak|Bj).
  \end{equation}
  
In particular, when $A=B=\emptyset$, we obtain the expression 
$\Delta(ik|j\ell)=\Delta(i|j)\Delta(k|\ell)-q\Delta(i|\ell)\Delta(k|j)$ (with $k=i+1$ and $\ell=j+1$), taking into account that $\Delta(\emptyset|\emptyset)=1$. This matches formula~\refeq{qminor} for the $q$-minor of a $2\times 2$ submatrix.
 
 (IV) \emph{Quasi-commutation relations on interval $q$-minors}. ~The simplest possible kind of universal QIs involves two corteges $(I|J),(I'|J')\in\Escr^{m,n}$ and is viewed as 
  \begin{equation} \label{eq:quasi-commut}
  \Delta(I|J)\Delta(I'|J')=q^c\Delta(I'|J')\Delta(I|J)
  \end{equation}
for some $c\in\Zset$. When $q$-minors $\Delta(I|J)$ and $\Delta(I'|J')$ satisfy~\refeq{quasi-commut}, they are called \emph{quasi-commuting}. (For example, three relations in~\refeq{xijkl} are such.) Leclerc and Zelevinsky~\cite{LZ} characterized such minors in the \emph{flag} case, by showing that $\Delta(I|[|I|])$ and $\Delta(I'|[|I'|])$ quasi-commute if and only if the subsets $I,I'$ of $[m]$ are \emph{weakly separated} (for a definition, see~\cite{LZ}). In a general case, a characterization of quasi-commuting $q$-minors is given in Scott~\cite{scott} (see also~\cite[Sect.~8.3]{DK} for additional aspects).

For purposes of this paper, it suffices to consider only \emph{interval $q$-minors}, i.e., assume that all $I,J,I',J'$ are intervals. Let for definiteness $|I|\ge|I'|$ and define 
  \begin{gather} 
  \alpha:=|\{i'\in I'\colon i'<\min(I)\}|,\qquad \beta:=|\{i'\in I'\colon i'>\max(I)\}|,
                                                                                            \label{eq:alphabeta} \\
  \gamma:=|\{j'\in J'\colon j'<\min(J)\}|,\qquad \delta:=|\{j'\in J'\colon j'>\max(J)\}|. \nonumber
   \end{gather}
Then the facts that $I,J,I',J'$ are intervals and that $|I|\ge|I'|$ imply $\alpha\beta=\gamma\delta =0$.  

Specializing Proposition~8.2 from~\cite{DK} to our case, we obtain that
  \begin{numitem1} \label{eq:qcom-int}
  for $|I|\ge |I'|$, interval $q$-minors $\Delta(I|J)$ and $\Delta(I'|J')$ quasi-commute (universally) if and only if $\alpha\gamma=\beta\delta=0$; in this case, $c$ as in~\refeq{quasi-commut} is equal to $\beta+\delta-\alpha-\gamma$.
    \end{numitem1}
     
In fact, we will use~\refeq{qcom-int} only  when $\Delta(I|J)$ is a flag or co-flag interval $q$-minor, and similarly for $\Delta(I'|J')$ (including mixed cases with one flag and one co-flag $q$-minors).

In this paper we explore the issue when the special quadratic identities exhibited in~(I)--(IV) determine all other universal QIs. More precisely, let $\Pscr=\Pscr_{m,n}$, $\Pscr^\ast=\Pscr^\ast_{m,n}$, and $\Dscr=\Dscr_{m,n}$ denote the sets of relations as in~\refeq{Pluck}, \refeq{co-Pluck}, and~\refeq{Dodgson}, respectively (concerning the corresponding objects in~(I)--(III)). Also let $\Qscr=\Qscr_{m,n}$ denote the set of quasi-commuting relations in (IV) concerning the flag and co-flag interval cases.
 \smallskip
 
\noindent\textbf{Definitions.} ~For $\Ascr,\,q,m,n$ as above, $f:\Escr^{m,n}\to \Ascr$ is called a \emph{QI-function} if its values satisfy the quadratic relations similar to those in the universal QIs on $q$-minors (i.e., when we formally replace $\Delta(I|J)$ by $f(I|J)$ in these relations). When  $f:\Escr^{m,n}\to \Ascr$ is assumed to satisfy the relations as in $\Pscr$, $\Pscr^\ast$ and $\Dscr$, we say that $f$ is an \emph{RQI-function} (abbreviating ``a function obeying \emph{restricted quadratic identities}'').
 \smallskip
 
Note that if $f:\Escr^{m,n}\to\Ascr$ satisfies a quadratic relation Q, and $a$ is an element of the center of $\Ascr$ (i.e. $ax=xa$ for any $x\in\Ascr$), then $af$ satisfies $Q$ as well. Hence if $f$ is a QI- or RQI-function, then so is $af$. Due to this, in what follows we will default assume that any function $f$ on $\Escr^{m,n}$ we deal with is \emph{normalized} , i.e., satisfies $f(\emptyset|\emptyset)=1$ (which is consistent with $\Delta(\emptyset|\emptyset)=1$).  
 
 Our goal is to prove two results on QI-functions. Let us say that a cortege $(I|J)\in\Escr^{m,n}$ is a \emph{double interval} if both $I,J$ are intervals. A double interval $(I|J)$ is called \emph{pressed} if at least one of $I,J$ is an initial interval, i.e., either $I=[|I|]$ or $J=[|J|]$ or both (yielding a flag or co-flag case); the set of these is denoted as $Pint=Pint_{m,n}$. 
 
 \begin{theorem} \label{tm:A}
Let RQI-functions $f,g:\Escr^{m,n}\to\Ascr-\{0\}$ coincide on $Pint_{m,n}$. Let, in addition, 
for any $(I|J)\in\Escr^{m,n}$, the element $f(I|J)$ is not a zerodivisor in $\Ascr$.
Then $f$ and $g$ coincide on the entire $\Escr^{m,n}$.
  \end{theorem}
  
It follows that any QI-function is uniquely determined by its values on $Pint$ and relations as in $\Pscr$, $\Pscr^\ast$ and $\Dscr$.

The second theorem describes a situation when taking values on $Pint$ arbitrarily within a representative part of $\Ascr$, one can extend these values to a QI-function (so one may say that, $Pint$ plays a role of ``basis'' for QI-functions, in a sense).

 \begin{theorem} \label{tm:B} Let $f_0:Pint\to\Ascr^\ast$ (where $\Ascr^\ast$ is the set of invertible elements of $\Ascr$). Suppose that $f_0$ satisfies the quasi-commutation relations (as in~\refeq{quasi-commut} in (IV)) on $Pint$. Then $f_0$ is extendable to a QI-function $f$ on $\Escr^{m,n}$.
  \end{theorem}
 
It should be noted that Theorems~\ref{tm:A} and~\ref{tm:B} can be regarded as quantum analogs of corresponding results in~\cite{DKK} devoted to universal quadratic identities on minors of matrices over a commutative semiring (e.g. over $\Rset_{>0}$ or over the tropical semiring $(\Rset,+,\max)$); see Theorem~7.1 there.
  \smallskip
  
This paper is organized as follows. Section~\SEC{proofA} contains a proof of Theorem~\ref{tm:A}. Section~\SEC{flows} reviews a construction, due to Casteels~\cite{cast}, used in our approach to proving the second theorem. According to this construction (of which idea goes back to Cauchon diagrams in~\cite{cauch}), the minors of a generic $q$-matrix can be expressed as the ones of the so-called \emph{path matrix} of a special planar graph $G_{m,n}$, viewed as an extended square grid of size $m\times n$. There is a one-to-one correspondence between the pressed interval corteges in $\Escr^{m,n}$ and the inner vertices of $G_{m,n}$. This enables us to assign each generator involved in the construction of entries of the path matrix (formed in Lindstr\H{o}m's style via path systems, or ``flows'', in $G_{m,n}$) as the ratio of two values of $f_0$; this is just where we use that $f_0$ takes values in $\Ascr^\ast$. 
Relying on this construction, we prove Theorem~\ref{tm:B} in Section~\SEC{proofB}; here the crucial step is to show that the quasi-commutation relations on the values of $f_0$ imply the relations on generators needed to obtain a corrected path matrix. Finally, in Section~\SEC{unique} we describe a situation when a function $f_0$ on $Pint_{m,n}$ exposed in Theorem~\ref{tm:B} has a unique extension to $\Escr^{m,n}$ that is a QI-function, or, roughly speaking, when the values on $Pint$ and relations as in $\Pscr,\Pscr,\Dscr$ and $\Qscr$ determine a QI-function on $\Escr^{m,n}$, thus yielding all other universal QIs.
  
 
 \section{Proof of Theorem~\ref{tm:A}} \label{sec:proofA}
 
Let $f,g:\Escr^{m,n}\to \Ascr$ be as in the hypotheses of this theorem. To show that $f(I|J)=g(I|J)$ holds everywhere, we consider three possible cases for $(I|J)\in\Escr^{m,n}$. In the first and second cases, we use induction on the value 
  $$
  \sigma(I,J):=\max(I)-\min(I)+\max(J)-\min(J).
  $$
  
\noindent \emph{Case 1}. Let $(I|J)$ be such that: (i) $f(I'|J')=g(I'|J')$ holds for all $(I'|J')\in\Escr^{m,n}$ with $\sigma(I',J')<\sigma(I,J)$; and (ii) $I$ is not an interval. 

Define $i:=\min(I)$, $k:=\max(I)$ and $A:=I-\{i,k\}$. Take $\ell\in J$ and let $B:=J-\ell$. Since $I$ is not an interval, there is $j\in[m]$ such that $i<j<k$ and $j\notin I$. Then $j\notin A$ and $(Aik|B\ell)=(I|J)$. Applying to $f$ and $g$ Pl\H{u}cker-type relations as in~\refeq{Pluck}, we have
   \begin{equation} \label{eq:fPluck}
  f(Aj|B)f(Aik|B\ell)=f(Aij|B\ell)f(Ak|B)+f(Ajk|B\ell)f(Ai|B), \quad\mbox{and}
  \end{equation}
   \begin{equation} \label{eq:gPluck}
  g(Aj|B)g(Aik|B\ell)=g(Aij|B\ell)g(Ak|B)+g(Ajk|B\ell)g(Ai|B).
  \end{equation}

The choice of $i,j,k,\ell$ provides that in these relations, the number $\sigma(A',B')$ for each of the five corteges $(A'|B')$ different from $(Aik|B\ell)$ ($=(I|J)$) is strictly less than $\sigma(I|J)$. So $f$ and $g$ coincide on these $(A'|B')$, by condition~(i) on $(I|J)$. Subtracting~\refeq{gPluck} from~\refeq{fPluck}, we obtain
  $$
  f(Aj|B)\left(f(I|J)-g(I|J)\right)=0.
  $$
This implies $f(I|J)=g(I|J)$ (since $f(Aj|B)\ne 0$ and $f(Aj|B)$ is not a zerodivisor, by the hypotheses of the theorem).
  \medskip
 
\noindent\emph{Case 2}. Let $(I|J)$ be subject to condition~(i) from the previous case and suppose that $J$ is not an interval. Then taking $i:=\min(J)$, $k:=\max(J)$, $B:=J-\{i,k\}$, $\ell\in I$, $A:=I-\ell$, applying to $f,g$ the corresponding co-Pl\H{u}cker-type relations as in~\refeq{co-Pluck}, and arguing as above, we again obtain $f(I|J)=g(I|J)$.
 \medskip
  
Thus, it remains to examine double intervals  $(I|J)$. We rely on the equalities $f(I|J)=g(I|J)$ when $(I|J)$ is pressed (belongs to $Pint$), and use induction on the value
   $$
  \eta(I,J):=\max(I)+\min(I)+\max(J)+\min(J).
  $$

\noindent\emph{Case 3}. Let $(I|J)\in\Escr^{m,n}$ be a non-pressed double interval. 
Define $i:=\min(I)-1$, $k:=\max(I)$, $j:=\min(J)-1$, $\ell:=\max(J)$, $A:=I-k$, $B:=J-\ell$. Then $i,j\ge 1$ (since $(I|J)$ is non-pressed). Also $(I|J)=(Ak|B\ell)$. Suppose, by induction, that $f(I'|J')=g(I'|J')$ holds for all double intervals $(I'|J')\in\Escr^{m,n}$ such that $\eta(I',J')<\eta(I,J)$.    

Applying to $f$ and $g$ Dodgson-type relations as in~\refeq{Dodgson}, we have
  \begin{equation} \label{eq:fDodg}
 f(Ai|Bj)f(Ak|B\ell)  =  f(Aik|Bj\ell)f(A|B)+ qf(Ai|B\ell)f(Ak|Bj), \quad\mbox{and}
  \end{equation}
  \begin{equation} \label{eq:gDodg}
 g(Ai|Bj)g(Ak|B\ell)  =  g(Aik|Bj\ell)g(A|B)+ qg(Ai|B\ell)g(Ak|Bj).
  \end{equation}

One can see that for all corteges $(A'|B')$ occurring in these relations, except for $(Ak|B\ell)$, the value $\eta(A',B')$ is strictly less than  $\eta(I,J)$. Therefore, subtracting~\refeq{gDodg} from~\refeq{fDodg} and using induction on $\eta$, we obtain
   $$
   f(Ai|Bj)\left(f(Ak|B\ell)-g(Ak|B\ell)\right)=0,
   $$
whence $f(I|J)=g(I|J)$, as required.

This completes the proof of the theorem. \hfill \qed

 
 \section{Flows in a planar grid} \label{sec:flows}
 
The proof of Theorem~\ref{tm:B} essentially relies on a construction of quantum minors via certain path systems (``flows'') in a special planar graph. This construction is due to Casteels~\cite{cast} and it was based on ideas in Cauchon~\cite{cauch} and Lindstr\H{o}m~\cite{lind}. Below we review details of the method needed to us, mostly following terminology, notation and conventions used for the corresponding special case in~\cite{DK}.
 \medskip
 
\noindent\textbf{Extended grids.} 
~Let $m,n\in\Zset_{>0}$. We construct a certain planar directed graph, called an \emph{extended $m\times n$ grid} and denoted as $G_{m,n}=G=(V,E)$, as follows.
 \smallskip
 
(G1) The vertex set $V$ is formed by the points $(i,j)$ in the plane $\Rset^2$ such that $i\in\{0\}\cup[m]$, $j\in\{0\}\cup [n]$ and $(i,j)\ne (0,0)$. Hereinafter, it is convenient to us to assume that the first coordinate $i$ of a point $(i,j)$ in the plane is the \emph{vertical} one.
  \smallskip
  
(G2) The edge set $E$ consists of edges of two types: ``horizontal'' edges, or \emph{H-edges}, and ``vertical'' edges, or \emph{V-edges}.
  \smallskip
  
(G3) The H-edges are directed from left to right and go from $(i,j-1)$ to $(i,j)$ for all $i=1,\ldots,m$ and $j=1,\ldots,n$.
  \smallskip
  
(G4) The V-edges are directed downwards and go from $(i,j)$ to $(i-1,j)$ for all $i=1,\ldots,m$ and $j=1,\ldots,n$.
  \smallskip
  
Two subsets of vertices in $G$ are distinguished: the set $R=\{r_1,\ldots,r_m\}$ of \emph{sources}, where $r_i:=(i,0)$, and the set $C=\{c_1,\ldots,c_n\}$ of \emph{sinks}, where $c_j:=(0,j)$. The other vertices are called \emph{inner} and the set of these (i.e.,  $[m]\times[n]$) is denoted by $W=W_{G}$.

The picture illustrates the extended grid $G_{3,4}$.

\vspace{0.3cm}
\begin{center}
\includegraphics{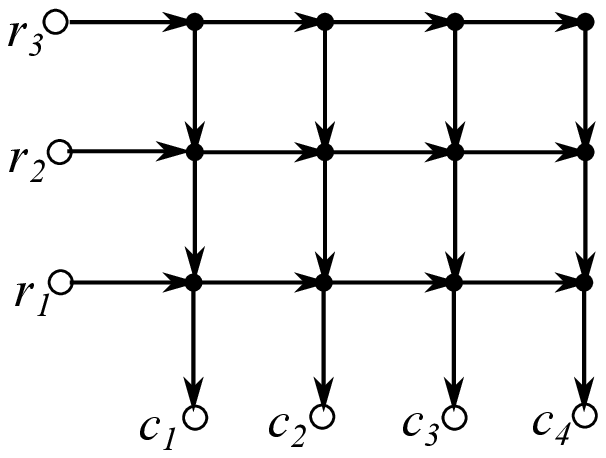}
\end{center}

Each inner vertex $v\in W$ of $G=G_{m,n}$ is regarded as a
\emph{generator}. This gives rise to assigning the \emph{weight} $w(e)$ to each edge $e=(u,v)\in E$ (going from a vertex $u$ to a vertex $v$) in a
way similar to that introduced for Cauchon graphs in~\cite{cast}, namely:
 \begin{numitem1} \label{eq:weight_edge}
   \begin{itemize}
\item[(i)] $w(e):=v$ if $e$ is an H-edge with $u\in R$;
\item[(ii)] $w(e):=u^{-1}v$ if $e$ is an H-edge and $u,v\in W$;
\item[(iii)] $w(e):=1$ if $e$ is a V-edge.
  \end{itemize}
  \end{numitem1}

This in turn gives rise to defining the weight $w(P)$ of a directed path
$P=(v_0,e_1,v_1,\ldots, e_k,v_k)$ (where $e_i$ is the edge from $v_{i-1}$ to $v_i$) to be the ordered (from left to right) product, namely:
  \begin{equation} \label{eq:wPe}
  w(P):=w(e_1)w(e_2)\cdots w(e_k).
  \end{equation}

Then $w(P)$ forms a Laurent monomial in elements of $W$. Note that when $P$ begins
in $R$ and ends in $C$, its weight can also be expressed in the following
useful form: if $u_1,v_1,u_2,v_2,\ldots, u_{d-1},v_{d-1},u_d$ is the sequence of vertices where $P$ makes turns (from ``east'' to ``sought'' at each $u_i$, and from ``sought'' to ``east'' at each $v_i$), then, due to the ``telescopic effect'' caused by~\refeq{weight_edge}(ii), there holds
  \begin{equation} \label{eq:telescop}
  w(P)=u_1v_1^{-1}u_2v_2^{-1}\cdots u_{d-1}v_{d-1}^{-1} u_d.
  \end{equation}

We assume that the elements of $W$ obey quasi-commutation laws which look somewhat
simpler than those in~\refeq{xijkl}; namely, for distinct inner vertices $u=(i,j)$ and $v=(i',j')$,
 \begin{numitem1} \label{eq:uvq}
   \begin{itemize}
\item[(i)] if $i=i'$ and $j<j'$, then $uv=qvu$;
\item[(ii)] if $i>i'$ and $j=j'$, then $vu=quv$;
\item[(iii)]  otherwise $uv=vu$,
   \end{itemize}
     \end{numitem1}
\noindent where, as before, $q\in {\mathbb K}^\ast$. (Note that $G$ has a horizontal (directed) path from $u$ to $v$ in~(i), and a vertical path from $u$ to $v$ in~(ii).)
  \medskip

\noindent\textbf{Path matrix and flows.} ~To be consistent with the vertex notation in extended grids, we visualize matrices in the Cartesian form: for an $m\times n$ matrix $A=(a_{ij})$, the row indexes $i=1,\ldots,m$ are assumed to grow upwards, and the column indexes $j=1,\ldots,n$ from left to right. 

Given an extended $m\times n$ grid $G=G_{m,n}=(V,E)$ with the corresponding partition $(R,C,W)$ of $V$ as above, we form the
\emph{path matrix} $\Path=\Path_G$ of $G$ in a spirit of~\cite{cast}; namely, $\Path$ is the $m\times n$ matrix whose entries are defined by
  \begin{equation} \label{eq:Mat}
  \Path(i|j):=\sum\nolimits_{P\in\Phi_G(i|j)} w(P), \qquad (i,j)\in [m]\times [n],
  \end{equation}
where $\Phi_G(i|j)$ is the set of (directed) paths from the source $r_i$ to the sink $c_j$ in $G$. Thus, the entries of $\Path_G$ belong to the $\mathbb K$-algebra $\Lscr_{G}$ of Laurent polynomials generated by the set $W$ if inner vertices of $G$ subject to~\refeq{uvq}.
 \medskip
 
\noindent \textbf{Definition.} ~
Let $(I|J)\in\Escr^{m,n}$. Borrowing terminology from~\cite{DKK}, 
by an $(I|J)$-\emph{flow} we mean a set $\phi$ of \emph{pairwise disjoint}
directed paths from the source set $R_I:=\{r_i\colon i\in I\}$ to the sink set
$C_J:=\{c_j\colon j\in J\}$ in $G$.
 \smallskip

The set of $(I|J)$-flows $\phi$ in $G$ is denoted by $\Phi(I|J)=\Phi_G(I|J)$.
We order the paths forming $\phi$ by increasing the indexes of sources: if $I$ consists of $i(1)<i(2)<\cdots< i(k)$ and $J$ consists of
$j(1)<j(2)<\cdots<j(k)$ and if $P_\ell$ denotes the path in $\phi$ beginning at
$r_{i(\ell)}$, then $P_\ell$ is just $\ell$-th path in $\phi$, $\ell=1,\ldots,k$. Note that the planarity of $G$ and the fact that the paths in $\phi$ are pairwise disjoint imply that each $P_\ell$ ends at the sink $c_{j(\ell)}$.

Similar to the assignment of weights for path systems in~\cite{cast}, we define the weight of $\phi=(P_1,P_2,\ldots,P_k)$ to be the ordered product
 \begin{equation} \label{eq:w_phi}
  w(\phi):=w(P_1)w(P_2)\cdots w(P_k).
  \end{equation}

Using a version of Lindstr\"om Lemma, Casteels showed a correspondence between path systems and $q$-minors of path matrices.
  \begin{prop}[\cite{cast}] \label{pr:Linds}
For the extended grid $G=G_{m,n}$ and any $(I|J)\in \Escr^{m,n}$, 
  \begin{equation} \label{eq:Lind}
\Delta(I|J)_{\Path_G,q}=\sum\nolimits_{\phi\in\Phi_G(I|J)} w(\phi).
  \end{equation}
  \end{prop}

\noindent (This is generalized to a larger set of graphs and their path matrices in~\cite[Theorem~3.1]{DK}.)

The next property, surprisingly provided by~\refeq{uvq}, is of most importance to us.
  \begin{prop} [\cite{cast}] \label{pr:Path-A}
The entries of $\Path_G$ obey Manin's relations (similar to those in~\refeq{xijkl}).
 \end{prop}

It follows that the $q$-minors of $\Path_G$ satisfy all universal QIs, and therefore, the function $g:\Escr^{m,n}\to\Lscr_{G}$ defined by $g(I|J):=\Path_G(I|J)$ is a QI-function. 

 
 \section{Proof of Theorem~\ref{tm:B}} \label{sec:proofB}
 
Let $f_0:Pint_{m,n}\to \Ascr^\ast$ be a function as in the hypotheses of this theorem. Our goal is to extend $f_0$ to a QI-function $f$ on $\Escr^{m,n}$. The idea of our construction is prompted by Propositions~\ref{pr:Linds} and~\ref{pr:Path-A}; namely, we are going to obtain the desired $f$ as the function of $q$-minors of an appropriate path matrix $\Path_G$ for the extended $m\times n$ grid $G=G_{m,n}$. 

For this purpose, we first have to determine the ``generators'' in $W$ in terms of values of $f_0$ (so as to provide that these values are consistent with the corresponding pressed interval $q$-minors of the path matrix), and second, using the quasi-commutation relations  (as in~\refeq{quasi-commut}) on the values of $f_0$, to verify validity of relations~\refeq{uvq} on the generators. Then $Path_G$ will be indeed a fine $q$-matrix and its $q$-minors will give the desired QI-function $f$.

(It should be emphasized that we may speak of a vertex of $G$ in two ways: either as a point in $\Rset^2$, or as a generator of the corresponding algebra. In the former case, we use the coordinate notation $(i,j)$ (where $i\in\{0\}\cup[m]$ and $j\in\{0\}\cup[n]$). And in the latter case, we use notation $w(i,j)$, referring to it as the \emph{weight} of $(i,j)$.)

To express the elements of $W$ via values of $f_0$, we associate each pair $(i,j)\in[m]\times[n]$ with the pressed interval cortege $\pi(i,j)=(I|J)$, where 
  \begin{numitem1} \label{eq:pi(i,j)}
$I:=[i-k+1..i]$ and $J:=[j-k+1..j]$, where $k:=\min\{i,j\}$.
  \end{numitem1}
 
 In other words, if $i\le j$ (i.e., $(i,j)$ lies ``south-east'' from the ``diagonal'' $\{\alpha,
\alpha\}$ in $\Rset^2$), then $(I|J)$ is the co-flag interval cortege with $I=[i]$ and $\max(J)=j$, and if $i\ge j$ (i.e., $(i,j)$ is ``north-west'' from the diagonal), then $(I|J)$ is the flag interval cortege with $\max(I)=i$ and $J=[j]$. Also it is useful to associate to $(i,j)$: the (almost rectangular) subgrid induced by the vertices in $(\{0\}\cup[i])\times (\{0\}\cup[j])-\{(0,0)\}$, and the \emph{diagonal} $D(i|j)$ formed by the vertices $(i,j),(i-1,j-1),\ldots, (i-k+1,j-k+1)$. See the picture where the left (right) fragment illustrates the case $i<j$ (resp. $i>j$), the subgrids  are indicated by thick lines, and the diagonals $D(i,j)$ by bold circles.

\vspace{-0cm}
\begin{center}
\includegraphics{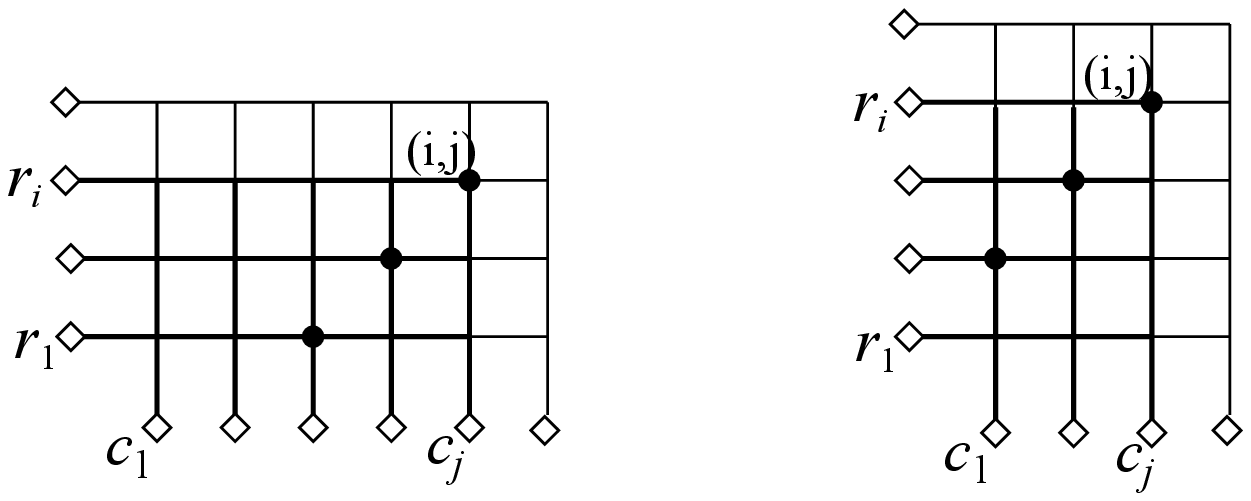}
\end{center}
\vspace{-0cm}

An important feature of a pressed interval cortege $(I|J)\in\Escr^{m,n}$ (which is easy to see) is that 
  \begin{numitem1} \label{eq:piflow}
$\Phi(I|J)$ consists of a unique flow $\phi$ and this flow is formed by paths $P_1,\ldots,P_k$, where for $i:=\max(I)$, $j:=\max(J)$, $k:=\min\{i,j\}$, and $\ell=1,\ldots,k$,
the path $P_\ell$ begins at $r_{i-k+\ell}$, ends at $c_{j-k+\ell}$ and makes exactly one turn, namely, the east to sought turn at the vertex $(i-k+\ell,j-k+\ell)$ of the diagonal $D(i|j)$.
  \end{numitem1}
  
We denote this flow $(P_1,\ldots,P_k)$ as $\phi(i|j)$; it is illustrated in the picture (for both cases $i<j$ and $i>j$ from the previous picture).

\vspace{-0cm}
\begin{center}
\includegraphics{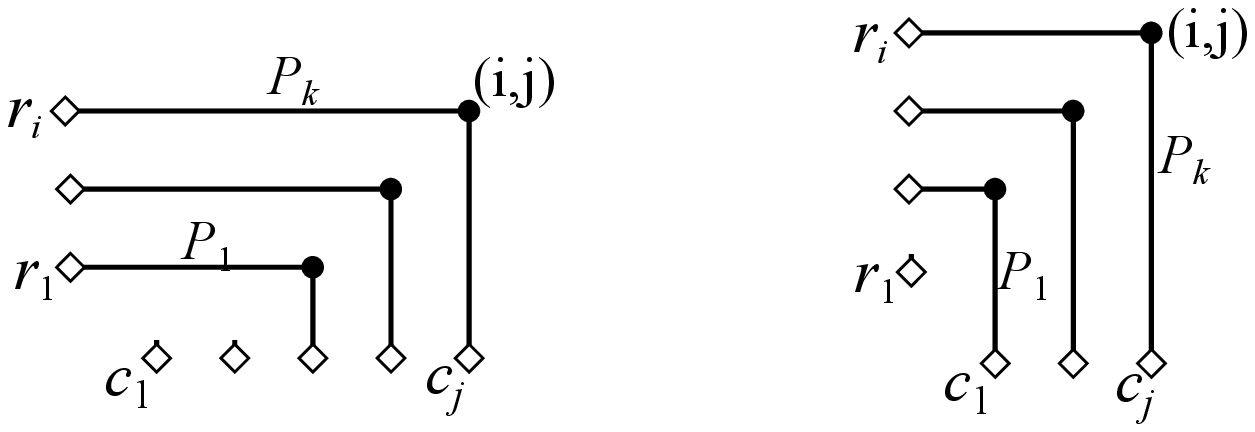}
\end{center}
\vspace{-0cm}

Therefore, for each $(i,j)\in[m]\times[n]$, taking the cortege $(I|J)=\pi(i,j)$ and the flow $\phi(i|j)=(P_1,\ldots,P_k)$ with $k=\min\{i,j\}$ and using expressions~\refeq{telescop} and~\refeq{w_phi}  for them, we obtain that
  \begin{equation} \label{eq:wphiij}
  \sum\nolimits_{\phi\in\Phi_G(I|J)} w(\phi)=w(\phi(i|j))=
       w(i-k+1,j-k+1)\cdots w(i-1,j-1)w(i,j).
      \end{equation}

Now imposing the conditions
  \begin{equation} \label{eq:wphi-f0}
  w(\phi(i|j)):=f_0(I|J) \quad \mbox{for all}\;\; (I|J)=\pi(i,j)\in Pint_{m,n},
    \end{equation}
we come to the rule of defining appropriate weights of inner vertices of $G$. Namely, relying on~\refeq{wphiij}, we define $w(i,j)$ for each $(i,j)\in[m]\times[n]$ by
  \begin{equation} \label{eq:wij2}
  w(i,j):=\left\{
    \begin{array}{rcl}
    f_0(\{i\}|\{j\}) \qquad \qquad\qquad &&\mbox{if}\;\; \min\{i,j\}=1, \\
     (f_0(\pi(i-1,j-1)))^{-1} f_0(\pi(i,j)) \quad &&\mbox{otherwise}.
    \end{array}
    \right.
     \end{equation}
Such a $w(i,j)$ is well-defined since $f_0(\pi(i-1,j-1))$ is invertible.

The crucial step in our proof is to show that these weights satisfy the relations as in~\refeq{uvq}, i.e., for $(i,j)$ and $(i',j')$,
  \begin{numitem1} \label{eq:wcast}
  \begin{itemize}
\item[(i)] if $i=i'$ and $j<j'$, then $w(i,j)w(i',j')=qw(i',j')w(i,j)$;
\item[(ii)]  if $i>i'$ and $j=j'$, then $w(i',j')w(i,j)=qw(i,j)w(i',j')$;
\item[(iii)] otherwize $w(i,j)w(i',j')=w(i',j')w(i,j)$.
  \end{itemize}
  \end{numitem1}
This would provide that $Path_G$ is indeed a fine $q$-matrix, due to~\refeq{Lind} and Proposition~\ref{pr:Path-A}, and setting $f(I|J):=\Delta(I|J)_{Path_G}$ for all $(I|J)\in \Escr^{m,n}$, we would obtain the desired function, thus completing the proof of the theorem.

First of all we have to explain that 
  \begin{numitem1} \label{eq:f0qc}
$f_0$ satisfies the quasi-commutation relation for any two pressed interval corteges $(I|J),(I'|J')\in Pint$, i.e., $f_0(I|J)f_0(I'|J')=q^c f_0(I'|J')f_0(I|J)$ holds for some $c\in\Zset$.
 \end{numitem1}
 
 This is equivalent to saying that such corteges determine a universal QI of the form~\refeq{quasi-commut} on associated $q$-minors. To see the latter, assume that $|I|\ge |I'|$ and define $\alpha,\beta,\gamma,\delta$ as in~\refeq{alphabeta}. One can check that: $\gamma=\delta=0$ if both interval corteges are flag ones; $\alpha=\beta=0$ if they are co-flag ones; and either $\beta=\gamma=0$ or $\alpha=\delta=0$ (or both) if one of these is a flag, and the other a co-flag interval cortege. So in all cases, we have $\alpha\gamma=\beta\delta=0$, and~\refeq{f0qc} follows from~\refeq{qcom-int}.
 
Next we start proving~\refeq{wcast}. Given $(i,j),(i',j')\in[m]\times[n]$, let $(I|J):=\pi(i,j)$ and $(I'|J'):=\pi(i',j')$, and define
   $$
   A:=f_0(I|J), \quad B:=f_0(I-i|J-j), \quad C:=f_0(I'|J'), \quad D:=f_0(I'-i'|J'-j'),
   $$
letting by definition $B:=1$ ($D:=1$) if $|I|=1$ (resp. $|I'|=1$). (Here for an element $p\in P$, we write $P-p$ for $P-\{p\}$.)

Then $w(i,j)$ is rewritten as $B^{-1}A$, and $w(i',j')$ as $D^{-1}C$ (by~\refeq{wij2}), and our goal is to show that 
  \begin{equation} \label{eq:BADC}
 B^{-1}AD^{-1}C=q^d D^{-1}CB^{-1}A,
  \end{equation}
where $d$ is as required in~\refeq{wcast} (i.e., equal to 1, -1, 0 in cases (i),(ii),(iii), respectively).

Define $c_1,c_2,c_3,c_4$ from the quasi-commutation relations (as in~\refeq{quasi-commut})
  \begin{equation} \label{eq:cccc}
  AC=q^{c_1}CA,\quad AD=q^{c_2}DA,\quad BC=q^{c_3}CB,\quad BD=q^{c_4}DB.
  \end{equation}
  
One can see that
  \begin{equation} \label{eq:d=c}
  d=c_1-c_2-c_3+c_4.
  \end{equation}
Indeed, in order to transform the string $ B^{-1}AD^{-1}C$ into $D^{-1}CB^{-1}A$, one should swap each of $A,B^{-1}$ with each of $C,D^{-1}$. The second equality in~\refeq{cccc} implies $AD^{-1}=q^{-c_2}D^{-1}A$, and for similar reasons, $B^{-1}C=q^{-c_3}CB^{-1}$ and $B^{-1}D^{-1}=q^{c_4}D^{-1}B^{-1}$. 
  \smallskip

Now we are ready to examine possible combinations for $(i,j)$ and $ (i',j')$ and compute $d$ in these cases by using~\refeq{d=c}. We will denote the intervals $I-i,\;J-j,\;I'-i',\; J'-j'$ in question by $\tilde I,\,\tilde J,\,\tilde I',\,\tilde J'$, respectively. Also for an ordered pair $((P|Q),(P'|Q'))$ of double intervals in $\Escr^{m,n}$ (where $|P'|=|Q'|$ may exceed $|P|=|Q|$), we define
  \begin{equation} \label{eq:PPQQ}
  \left.
    \begin{array}{rcl}
    \alpha(P,P')&:=& \min\{|\{p'\in P'\colon p'<\min(P)\}|,\; |\{p\in P\colon p>
                                                   \max (P')\}|\}; \\
    \beta(P,P')&:=& \min\{|\{p'\in P'\colon p'>\max(P)\}|,\; |\{p\in P\colon p<
                                                   \min (P')\}|\},
    \end{array}
    \right.
     \end{equation}
and define $\gamma(Q,Q')$ and $\delta(Q,Q')$ in a similar way (this matches the definition of $\alpha,\beta,\gamma,\delta$ in~\refeq{alphabeta} when $|P|\ge|P'|$). 
Using~\refeq{qcom-int}, we observe that the sum $\beta(I,I')+\delta(J,J')-\alpha(I,I')-\gamma(J,J')$ is equal to $c_1$, and similarly for the pairs concerning $c_2,c_3,c_4$. 

In our analysis we also will use the values
 \begin{gather} 
 \varphi:= (\beta(I,I')-\alpha(I,I'))-(\beta(I,\tilde I')-\alpha(I,\tilde I'))
   -(\beta(\tilde I,I')-\alpha(\tilde I,I'))+(\beta(\tilde I,\tilde I')-\alpha(\tilde I,\tilde I'));
                                             \nonumber \\
 \psi:= (\delta(J,J')-\gamma(J,J'))-(\delta(J,\tilde J')-\gamma(J\tilde J'))
  -(\delta(\tilde J,J')-\gamma(\tilde J,J'))+(\delta(\tilde J,\tilde J')-\gamma(\tilde J,\tilde J')).
                                                    \nonumber
  \end{gather}

In view of~\refeq{d=c} and~\refeq{PPQQ},
  \begin{equation} \label{eq:ccccd}
    \varphi+\psi=c_1-c_2-c_3+c_4=d.
  \end{equation}
  
The lemmas below compute $\varphi$ using~\refeq{PPQQ}. Let $r:=\min(I)$ ($=\min(\tilde I)$) and $r':=\min(I')$ ($=\min(\tilde I')$). 

 \begin{lemma} \label{lm:41}
 Suppose that $|I|\ne|I'|$ and $i\ne i'$. Then $\varphi=0$.
   \end{lemma}
   \begin{proof}
Assume that $|I|>|I'|$. Then $|I|>|\tilde I|\ge|I'|>|\tilde I'|$. Consider possible cases.
  \medskip

\noindent\emph{Case 1}:  $r\le r'$ and $i'<i$. Then $I',\tilde I'\subseteq I,\tilde I$. Therefore, both $\alpha$ and $\beta$ are zero everywhere, implying $\varphi=0$.
\medskip

\noindent\emph{Case 2}: $I\cap I'=\emptyset$. If $i'< r$, then $\beta$ is zero. Also  $ \alpha(I,I')=|I'|=\alpha(\tilde I,I')$ and $\alpha(I,\tilde I')=|\tilde I'|=\alpha(\tilde I,\tilde I')$.
    
And if $i< r'$, then $\alpha$ is zero. Also $\beta(I,I')=|I'|=\beta(\tilde I,I')$ and $\beta(I,\tilde I')=|\tilde I'|=\beta(\tilde I,\tilde I')$.
So in both situations, $\varphi=0$.
  \medskip

\noindent\emph{Case 3}: $r'<r\le i'<i$. Then $\beta$ is zero. Also $\alpha(P,P')=r-r'$ holds for all $P\in\{I,\tilde I)$ and $P'\in\{I',\tilde I'\}$, implying $\varphi=0$.
  \medskip

\noindent\emph{Case 4}: $r<r'\le i<i'$. Then $\alpha$ is zero, and
  $$
  \beta(I,I')=i'-i=\beta(\tilde I,\tilde I'), \quad \beta(I,\tilde I')=i'-1-i \quad \mbox{and}\quad \beta(\tilde I,I')=i'-(i-1),
  $$
again implying $\varphi=0$.

When $|I|<|I'|$, the argument follows by swapping $I,\tilde I$ by $I',\tilde I'$.
  \end{proof}

 \begin{lemma} \label{lm:42} Let  $|I|=|I'|$. {\rm(a)} If $i<i'$ then $\varphi=1$. {\rm(b)}  If $i>i'$ then $\varphi=-1$. {\rm(c)}  If $i=i'$ then $\varphi=0$.  
   \end{lemma}
   \begin{proof}
~We have $|I'|,|\tilde I'|\le |I|$ and $|\tilde I'|=|\tilde I|$ but $|I'|=|\tilde I|+1$. Let $i>i'$. Then, using~\refeq{PPQQ}), one can check that $\beta$ is zero. Also if $I\cap I'=\emptyset$, then 
  $$
  \alpha(I,I')=|I|,\quad \alpha(I,\tilde I')=|\tilde I'|=\alpha(\tilde I,\tilde I'), \quad \alpha(\tilde I,I')=|\tilde I|=|I|-1.
    $$ 
And if $I\cap I'\ne\emptyset$, then
  $$
  \alpha(I,I')=\alpha(I,\tilde I')=\alpha(\tilde I,\tilde I')=r-r'=i-i'\quad \mbox{and} \quad \alpha(\tilde I,I')=|\tilde I-I'|=(i-1)-i'.
    $$ 

Therefore, in both situations
  $$
  \varphi=-\alpha(I,I')+\alpha(I,\tilde I')+\alpha(\tilde I,I')-\alpha(\tilde I,\tilde I')=
    \alpha(\tilde I,I')-\alpha(I,I')=-1,
  $$ 
as required in~(b).

Case (a) reduces to (b). And if $i=i'$ then $r=r'$, implying that both $\alpha, \beta$ are zero (since for any two intervals among $I,\tilde I,I',\tilde I'$, one is included in the other).
 \end{proof}
 
 \begin{lemma} \label{lm:43}
 Let $i=i'$. {\rm(a)} If $|I|>|I'|$ then $\varphi=-1$. {\rm(b)} If $|I|<|I'|$ then $\varphi=1$.
  \end{lemma}
 \begin{proof}
~Let $|I|>|I'|$. Then $I',\tilde I\subset I$ and $\tilde I'\subset \tilde I$. Hence $\alpha$ and $\beta$ are zero on each of $(I|I'),(I|\tilde I'),(\tilde I|\tilde I')$. Also $|\tilde I|\ge |I'|$ and $r<r'$ imply $\alpha(\tilde I,I')=0$ and $\beta(\tilde I,I')=i'-(i-1)=1$ (since $\max(\tilde I)=i-1$). This gives $\varphi=-\beta(\tilde I,I')=-1$. 

Case~(b) reduces to~(a). 
 \end{proof}

Replacing   $i,i'$ by $j,j'$, and $I,I'$ by $J,J'$ in Lemmas~\ref{lm:41}--\ref{lm:43}, we obtain the corresponding statements concerning $\psi$.
  \begin{numitem1} \label{eq:psi}
  \begin{itemize}
\item[(i)] If $|J|=|J'|$ and $j<j'$, or if $|J|<|J'|$ and $j=j'$, then $\psi=1$.
\item[(ii)] Symmetrically, if $|J|=|J'|$ and $j>j'$, or if $|J|>|J'|$ and $j=j'$, then $\psi=-1$.
\item[(iii)] Otherwise $\psi=0$.
  \end{itemize}
  \end{numitem1}
  
Now we finish the proof with showing~\refeq{wcast} in the corresponding three cases.
  \medskip
  
\noindent\emph{Case A}: $i=i'$ and $j<j'$. First suppose that $i\le j$. Then both $(I|J)$ and $(I'|J')$ are co-flag corteges, and $|I|=|I'|=i$. We have $\varphi=0$ (by Lemma~\ref{lm:42}(c)) and $\psi=1$ (by~\refeq{psi}(i)).

Next suppose that $j<i<j'$. Then $(I|J)$ is flag, $(I'|J')$ is co-flag, and $|I|=j<i=|I'|$. This gives $\varphi=1$ (by Lemma~\ref{lm:43}(b)) and $\psi=0$ (by~\refeq{psi}(iii)).

Finally, suppose that $j'\le i$. Then both $(I|J),(I'|J')$ are flag, and $|I|=j<j'=|I'|$. This gives $\varphi=1$ (by Lemma~\ref{lm:43}(b)) and $\psi=0$ (by~\refeq{psi}(iii)).

Thus, in all situations, $d=\varphi+\psi=1$, as required in~\refeq{wcast}(i).
   \medskip
  
\noindent\emph{Case B}: $i<i'$ and $j=j'$. This is symmetric to the previous case, yielding $d=1$. This matches assertion~(ii) in~\refeq{wcast} (since replacing $i<i'$ by $i>i'$ changes $d=1$ to $d=-1$).
   \medskip
  
\noindent\emph{Case C}: $i\ne i$ and $j\ne j'$. When $\varphi=\psi=0$, \refeq{wcast}(iii) is immediate. The situation with $\varphi\ne 0$ arises only when $|I|=|I'|$; then (a) $i<i'$ implies $\varphi=1$, and (b) $i>i'$ does $\varphi=-1$ (see Lemma~\ref{lm:42}). Similarly, $\psi\ne 0$ happens only if $|J|=|J'|$; then (c) $j<j'$ implies $\psi=1$, and (d) $j>j'$ does $\psi=-1$ (by~\refeq{psi}(i),(ii))

In subcase~(a), $i<i'$ and $|I|=|I'|=:k$ imply $i'>k$ (in view of $i\ge |I|$). Therefore, $j'=k$ must hold (i.e., $(I'|J')$ is flag). Then $j\ne j'$ implies $j>j'$, and we obtain $\psi=-1$, by~\refeq{psi}(ii).

In subcase~(b), $i>i'$ and $|I|=|I'|=:k$ imply $i>k$. Therefore, $j=k$. Then $j'>j$, yielding $\psi=1$, by~\refeq{psi}(i). 

So in both (a) and (b), we obtain $\varphi+\psi=0$. In their turn, subcases~(c) and~(d) are symmetric to~(a) and~(b), respectively. Thus, in all situations, $d=0$ takes place, as required in~\refeq{wcast}(iii).
 \medskip
  
This completes the proof of Theorem~\ref{tm:B}.


\section{Uniqueness} \label{sec:unique}

Let $f_0:Pint_{m,n}\to\Ascr^\ast$ be a function in the hypotheses of Theorem~\ref{tm:B}, i.e., $f_0$ satisfies quasi-commutation relations for all pairs of pressed interval corteges in $\Escr^{m,n}$ (cf.~\refeq{f0qc}). A priori, $f_0$ may have many extensions to $\Escr^{m,n}$ that are QI-functions. One of them is the function $f$ whose values $f(I|J)$ are $q$-minors $\Delta(I|J)$ of the corresponding path matrix constructed in the proof  in Sect.~\SEC{proofB}. 

In light of Theorems~\ref{tm:A} and~\ref{tm:B}, it is tempting to ask when $f_0$ has a unique QI-extension. Since any QI-extension is an RQI-function (i.e., satisfies the corresponding relations of Pl\H{u}cker, co-Pl\H{u}cker and Dodgson types) and in view of Theorem~\ref{tm:A}, we may address an equivalent question: when an RQI-extension $g$ of $f_0$ is a QI-function (and therefore $g=f$). We give sufficient conditions below (which is, in fact, a corollary of Theorems~\ref{tm:A} and~\ref{tm:B}).

To this aim, let us associate to each $(I|J)\in Pint_{m,n}$ an indeterminate $y_{I|J}$ and form the $\mathbb K$-algebra $\Lscr_Y$ of quantized Laurent polynomials generated by these $y_{I|J}$ (where the quantization is agreeable with that for $f_0$). The values of $f_0$ are said to be \emph{algebraically independent} if the map $y_{I|J}\mapsto f_0(I|J)$, $(I|J)\in Pint_{m,n}$, gives an isomorphism between $\Lscr_Y$ and the $\mathbb K$-subalgebra $\Ascr^{f_0}$ of $\Ascr$ generated by these values.
  \begin{coro} \label{cor:unique}
Let $f_0$ and $f$ be as above. Let the following additional conditions hold:
 
{\rm(i)} the values of $f_0$ are algebraically independent;

{\rm(ii)} if an element $a\in\Ascr^{f_0}$ is a zerodivisor in $\Ascr$, then $a$ is a zerodivisor in $\Ascr^{f_0}$.

Suppose that $g$ is an RQI-function on $\Escr^{m,n}$ coinciding with $f_0$ on $Pint_{m,n}$. Then $g$ is a QI-function (and therefore $g=f$).
  \end{coro} 
  
  \begin{proof}
(a sketch)~Considering the construction of $q$-minors of the path matrix related to $f_0$ (cf.~\refeq{Mat},\refeq{Lind},\refeq{wphiij}--\refeq{wij2}), one can deduce that for each cortege $(I|J)\in\Escr^{m,n}$, ~$y_{I|J}$ is a nonzero polynomial in $\Lscr_Y$. Then condition~(i) implies that $f(I|J)$ is a nonzero element of $\Ascr^{f_0}$. Furthermore, since $\Lscr_Y$ is free of zerodivisors (by a known fact on Laurent polynomials; see, e.g.~\cite{bourb}, ch.~II, \$11.4, Prop.~8), so is $\Ascr^{f_0}$. Therefore, by condition~(ii), $f(I|J)$ is not a zerodivisor in $\Ascr$. Now applying Theorem~\ref{tm:A}, we obtain $g=f$, as required.
\end{proof}

\end{document}